\documentclass[11pt,oneside]{article}
\usepackage[OT4]{fontenc}\usepackage[cp1250]{inputenc}\usepackage{a4} \usepackage{amsfonts}\usepackage{amsmath}\usepackage{amssymb}\usepackage{amsxtra}
\usepackage{amsthm}\usepackage{a4}

\usepackage[dvips]{graphicx}\usepackage{pstricks}\usepackage{color}
\usepackage{bezier}\usepackage{epic}\usepackage{eepic}
\usepackage{pstricks}\usepackage{pst-node}\usepackage{lineno}
\usepackage[normalem]{ulem}\usepackage{epsf}

\newtheorem{thm}{Theorem}[section]

\newtheorem{prop}[thm]{Proposition}

\newtheorem{cor}[thm]{Corollary}

\let\oldenumerate\enumerate
\renewcommand{\enumerate}{
  \oldenumerate
  \setlength{\itemsep}{0.5pt}
  \setlength{\parskip}{0pt}
  \setlength{\parsep}{0pt}
}

\begin{document}

\title{Characterization of $\alpha$-excellent $2$-trees}
\author{$^1$Magda Dettlaff, $^{2}$Michael A. Henning\thanks{Research
supported in part by the University of Johannesburg}  \, and  $^{3}$Jerzy Topp\\
\\
$^1$Faculty of Mathematics, Physics and Informatics\\
University of Gda\'nsk \\
80-952 Gda\'nsk, Poland \\
\small \tt magda.dettlaff@ug.edu.pl\\
\\
$^2$Department of Mathematics and Applied Mathematics\\ University of Johannesburg \\ Auckland Park 2006, South Africa \\ \small {\tt mahenning@uj.ac.za} \\
\\
$^3$Institute of Applied Informatics\\ University of Applied Sciences\\ 82-300 Elbl\c{a}g, Poland \\ \small {\tt j.topp@ans-elblag.pl}}

\date{}
\maketitle

\begin{abstract}
A graph is $\alpha$-excellent if every vertex of the graph is contained in some maximum independent set of the graph. In this paper, we present two characterizations of the $\alpha$-excellent $2$-trees.
\end{abstract}

{\small \textbf{Keywords:} independence number; excellent graph; $k$-tree.} \\
\indent {\small \textbf{AMS subject classification: 05C69, 05C85}}

\section{Introduction}

For notation and graph theory terminology we, in general, follow the recent books~\cite{HaHeHe-20,HaHeHe-21,HaHeHe-22}. Specifically, let $G=(V(G),E(G))$ be a~graph with vertex set $V(G)$ and edge set $E(G)$. For a~vertex $v$ of $G$, its \emph{neighborhood\/}, denoted by $N_{G}(v)$, is the set of all vertices adjacent to $v$. The \emph{closed neighborhood\/} of $v$, denoted by $N_{G}[v]$, is the set $N_{G}(v)\cup \{v\}$.  For a set $S \subseteq V(G)$, the  \emph{open neighborhood} of $S$ is the set $N_G(S) = \bigcup_{v \in S} N_G(v)$, and the \emph{closed neighborhood} $N_G[S] = N_G(S) \cup S$. For a positive integer $k$, we let $[k] = \{1, \ldots, k\}$.

A subset $D$ of the vertex set $V(G)$ of a graph $G$ is called a \emph{dominating set} of $G$ if every vertex belonging to $V(G)\setminus D$ is adjacent to at least one vertex in $D$. A subset $I$ of $V(G)$ is \emph{independent} if no two vertices belonging to $I$ are adjacent in $G$. The cardinality of a largest (i.e., maximum) independent set of $G$, denoted by $\alpha(G)$, is called the \emph{independence number} of $G$. Every largest independent set of a graph is called an $\alpha$-set of the graph. The \emph{independent domination number} of $G$, denoted by $i(G)$, is the cardinality of the smallest independent dominating set of $G$ (or equivalently, the cardinality of a  minimum maximal independent set of vertices in $G$). The \emph{common independence number} of a graph $G$, denoted by $\alpha_c(G)$, is the greatest integer $r$ such that every vertex of $G$ belongs to some independent subset $X$ of $V(G)$ with $|X| \ge r$. It follows immediately from the above definitions that the common independence number is bounded below by the independent domination number and above by the independence number. Formally, for any graph $G$,
\begin{equation}
\label{obser-1} \tag{0} i(G) \le \alpha_c(G) \le \alpha(G).
\end{equation}

The study of independent sets in graphs was begun by Berge \cite{Berge1962, Berge1981} (see also \cite{Berge1985}) and Ore \cite{Ore1962}. We refer the reader to the book~\cite{HaHeHe-22} and to the survey \cite{GoddardHenning} of results on independent domination in graphs published in 2013 by Goddard and Henning. A graph $G$ is said to be \emph{well-covered\/} if $i(G)=\alpha(G)$. Equivalently, $G$ is well-covered if every maximal independent set of $G$ is a maximum independent set of $G$. The concept of well-covered graphs was introduced by Plummer \cite{Plummer1} in 1970. Since then the well-covered graphs were very extensively investigating in many papers. We refer the reader to the excellent (but already older) survey on well-covered graphs by Plummer \cite{Plummer}. We are interested in characterization of \emph{$\alpha$-excellent graphs}, that is, graphs $G$ for which $\alpha_c(G) = \alpha(G)$. Thus, if $G$ is an $\alpha$-excellent graph, every vertex belongs to some $\alpha$-set of $G$. It follows from Inequalities~(\ref{obser-1}) that every well-covered graph is an $\alpha$-excellent graph. The example of the cycle $C_6$ shows that the set of well-covered graphs is properly contained in the set of $\alpha$-excellent graphs.
The $\alpha$-excellent graphs have been studied in \cite{Dettlaff-Henning-Topp-alpha-excellent-bipartite, DLT, Domke, Fricke, Pushpalatha, Pushpalatha2} and in \cite{Berge1962,  Berge1981} as $B$-graphs.
In this paper, we begin the study of $\alpha$-excellent $k$-trees.

\section{Preliminary results}

A vertex $v$ of a graph $G$ is a \emph{simplicial vertex} if every two vertices belonging to $N_G(v)$ are adjacent in $G$. Equivalently, a  simplicial vertex is a vertex that  appears in exactly one clique of a graph, where a~\emph{clique} of a graph $G$ is a maximal complete subgraph of $G$. A clique of a~graph $G$ containing at least one simplicial vertex of $G$ is called a \emph{simplex} of $G$. Note that if $v$ is a simplicial vertex of $G$, then $G[N_G[v]]$ is the unique simplex of $G$ containing $v$. We begin with a simple proposition.

\begin{prop} \label{simplicies-are-pairwise-vertex-disjoint}
No $\alpha$-excellent graph contains a vertex belonging to at least two its simplexes.
\end{prop}
\begin{proof} Assume that a vertex $v$ of a graph $G$ belongs to two simplicies of $G$, say to $G[N_G[u]]$ and $G[N_G[w]]$. If $I$ is a maximum independent set that contains $v$, then $(I \setminus \{v\}) \cup \{u,w\}$ is an independent set of greater cardinality. Thus, $\alpha(G) \ge |(I \setminus \{v\}) \cup \{u,w\}|> |I|$, implying that $v$ does not belong to any $\alpha$-set of $G$, and proving that $G$ is not an $\alpha$-excellent graph.
\end{proof}

For a positive integer $k$, a graph $G$ is called a $k$-\emph{tree} if it can be
obtained from the complete graph $K_k$ by a finite number of applications of
the following operation: add a new vertex and join it to $k$ mutually adjacent
vertices of the existing graph. Certainly, every $1$-tree is a tree and vice
versa. In \cite{Rose}, Rose proved that a graph $G$ is a $k$-tree if and only if
the following conditions are fulfilled: $(i)$ $G$ is connected, $(ii)$ $G$
contains $K_k$ as a subgraph and does not contain $K_{k+2}$ as a subgraph,
$(iii)$ if $v$ and $u$ are nonadjacent vertices of $G$, then the subgraph
induced by the smallest $v\!-\!u$ separator is a complete graph on $k$
vertices. Note that $K_k$ and $K_{k+1}$ are the only $k$-trees of order $k$
and $k+1$, respectively.

It was proved in \cite{Dettlaff-Henning-Topp-alpha-excellent-bipartite} that a bipartite graph (and, in particular, a tree) is an $\alpha$-excellent graph if and only if it has a perfect matching. In this paper we are interested in possible extensions of that characterization to a characterization of $\alpha$-excellent $k$-trees for every positive integer $k$. We begin with the following definition.

A set ${\cal P}$ of complete subgraphs of a graph $G$ is said to be a \emph{perfect $(k+1)$-cover} of $G$ if each subgraph belonging to $\cal P$ is of order $k+1$ and every vertex of $G$ belongs to exactly one subgraph in ${\cal P}$. It is obvious that for $k=1$ there exists a one-to-one correspondence between perfect 2-covers of a graph and perfect matchings of the graph. Here we are interested in the existence of perfect $(k+1)$-covers in $k$-trees. First of all, one can prove that every $k$-tree has at most one perfect $(k+1)$-cover. In the following proposition we present the first relationship between $k$-trees having perfect $(k+1)$-covers and $\alpha$-excellent graphs.

\begin{prop} \label{perfect k-trees which are alpha excellent} If a $k$-tree $G$ has a perfect $(k+1)$-cover, then $G$ is an $\alpha$-excellent graph.
\end{prop}
\begin{proof} Let $G$ be a connected $k$-tree of order $n\ge k+1$. Then $G$ is a
$(k+1)$-partite graph, say $A_1, A_2, \ldots, A_{k+1}$ are partite sets of $G$ and
assume that $|A_1|\ge |A_2|\ge\ldots\ge |A_{k+1}|\ge 1$. In addition, since  $A_1, A_2,
\ldots, A_{k+1}$  are independent sets of vertices and $|A_1|+|A_2|+\cdots +|A_{k+1}|=
n$, it follows that $|A_1|\ge n/(k+1)$, and, therefore, $\alpha(G)\ge |A_1|\ge n/(k+
1)$. Let $I$ be an $\alpha$-set of $G$. Assume now that ${\cal P} = \{P_1,\ldots,
P_{\ell}\}$ is a perfect $(k+1)$-cover of $G$. Then, ${\ell}=n/(k+1)$, $|I\cap V({P_i})|\le 1$ for each $i \in [\ell]$,  and
\[
\alpha(G) = |I| = |I \cap \bigcup_{i=1}^{\ell} V({P_i})| = \sum_{i=1}^{\ell}|I\cap V({P_i})| \le \ell =  \frac{n}{k+1}.
\]
Consequently, $|A_1|= |A_2|= \cdots = |A_{k+1}|= n/(k+1)= \alpha(G)$, and each of the sets $A_1, A_2, \ldots, A_{k+1}$ is an $\alpha$-set of $G$. This implies that every vertex of $G$ belongs to an $\alpha$-set of $G$ and, therefore,  $G$ is an $\alpha$-excellent graph.
\end{proof}

\section{$\alpha$-Excellent $2$-trees}

From Proposition \ref{perfect k-trees which are alpha excellent} we know that a $k$-tree having a perfect $(k+1)$-cover is an $\alpha$-excellent $k$-tree. It is not clear to us whether the converse of this statement is true. That is, we do not know if every $\alpha$-excellent $k$-tree of order at least $k+1$ has a perfect $(k+1)$-cover if $k\ge 3$. However when $k = 2$, we provide in this paper a characterization of $\alpha$-excellent $k$-trees. For notational simplicity, in what follows if three vertices $a$, $b$, and $c$ are mutually adjacent in a~graph $G$, then the induced subgraph $G[\{a, b, c\}]$ of $G$ is isomorphic to $K_3$ and is called a \emph{triangle} in $G$, and we simply write $abc$ rather than $G[\{a, b, c\}]$. To every triangle in a graph $G$, we assign label $R$ or $B$ (as red or blue, resp.), and by $R(G)$ and $B(G)$ we denote the set of all triangles in $G$ that have label $R$ and $B$, respectively.  We also say that $R(G)$ and $B(G)$ are the sets of all ``red'' and ``blue'' triangles in $G$, respectively.\\ [-12pt]

We are now in position to present a constructive characterization of $\alpha$-excellent $2$-trees. For this purpose, let $\cal E$ be the family of labeled $2$-trees defined recursively as follows: \\ [-22pt]
\begin{enumerate}
\item[{\rm (1)}] The family $\cal E$ contains the $2$-tree of order~$3$ in which the only triangle is red, that is, it has label $R$. \\[-12pt]
\item[{\rm (2)}] The family $\cal E$ is closed under the operations ${\cal O}_1$ and ${\cal O}_2$ defined below:
\begin{itemize}
\item \textbf{Operation ${\cal O}_1$.}
   If a graph $G'$ belongs to ${\cal E}$ and $v_1v_2$ is an edge of $G'$, then ${\cal O}_1= {\cal O}_1(v_1,v_2)$ forms a graph $G$ by adding three new vertices $u_1$, $u_2$, $u_3$ to $G'$ in such a way that $v_1v_2u_1$, $v_2u_1u_2$ and $u_1u_2u_3$ are three new triangles, while $R(G)= R(G')\cup \{u_1u_2u_3\}$ and $B(G)= B(G') \cup \{v_1v_2u_1, v_2u_1u_2\}$. In this case we say that we apply the operation ${\cal O}_1$ to the edge $v_1v_2$ of $G'$.
\item \textbf{Operation ${\cal O}_2$.}
    If a graph $G'$ belongs to ${\cal E}$, $v_1v_2v_3$ is a red triangle in $G'$ (that is, $v_1v_2v_3\in R(G')$), and $v_4$ is a neighbor of $v_3$ (it is possible that $v_4\in \{v_1,v_2\} \subseteq N_{G'}(v_3)$), then ${\cal O}_2= {\cal O}_2(v_1v_2,v_3v_4)$ forms a graph $G$ by adding to $G'$  three new vertices $u_0$, $u_1$ and $u_2$ in such a way that $u_0v_1v_2$, $v_3v_4u_1$, and $v_3u_1u_2$ are new triangles, while $R(G)= (R(G')\setminus \{v_1v_2v_3\})\cup \{u_0v_1v_2, v_3u_1u_2\}$ and $B(G)= B(G') \cup \{v_1v_2v_3, v_3v_4u_1\}$. In this case we say that we apply the operation ${\cal O}_2$ to the edge $v_1v_2$ of the triangle $v_1v_2v_3$ and to the edge $v_3v_4$.
\end{itemize}
\end{enumerate}

The operations ${\cal O}_1$ and ${\cal O}_2$ are illustrated in Fig. \ref{Operations}. Note that the operation ${\cal O}_2$  changes ``colors'' of certain triangles, and the red triangle $v_1v_2v_3$ in $G'$ is recolored blue in $G$.

\begin{figure}[h!] \begin{center} \hspace*{2ex}\includegraphics[scale=1.25]{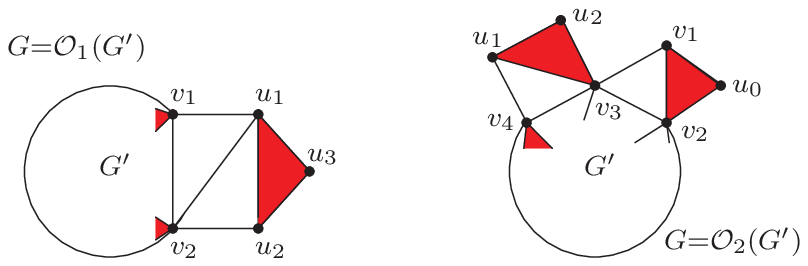}
\caption{The operations ${\cal O}_1$ and ${\cal O}_2$} \label{Operations} \end{center}\end{figure}

From the recursive definition of the graphs belonging to the family $\cal E$ it follows readily that if a $2$-tree $G$ belongs to $\cal E$, then the set $R(G)$ of red triangles in $G$ is a perfect $3$-cover of $G$. From this and from Proposition \ref{perfect k-trees which are alpha excellent} it follows that $G$ is an $\alpha$-excellent graph. Thus we have the following proposition that we will need when proving our main theorem.

\begin{prop} \label{2-trees belonging to the family E}
Every $2$-tree belonging to the family $\cal E$ has a perfect $3$-cover and it is an $\alpha$-excellent graph.
\end{prop}

The next theorem is the main result of this paper and it presents two characterizations of the $\alpha$-excellent $2$-trees: a constructive characterization, and a characterization in terms of perfect $3$-covers.

\begin{thm} \label{excellent-2-trees}
If $G$ is a $2$-tree of order $n\ge 3$, then the following statements are equivalent: \vspace*{-3ex}
\begin{enumerate}
\item $G$ has a perfect $3$-cover.
\item $G$ belongs to the family ${\cal E}$.
\item $G$ is an $\alpha$-excellent graph.
\end{enumerate}
\end{thm}
\begin{proof} The implications $(\rm a) \Rightarrow (\rm c)$, $(\rm b) \Rightarrow (\rm a)$, and $(\rm b) \Rightarrow (\rm c)$ are obvious from Propositions \ref{perfect k-trees which are alpha excellent} and \ref{2-trees belonging to the family E}. Thus it suffices to prove the implication $(\rm c) \Rightarrow (\rm b)$ (but we prove the implications  $(\rm c) \Rightarrow (\rm a)$ and $(\rm c) \Rightarrow (\rm b)$ at the same time).

Thus assume that $G$ is an $\alpha$-excellent $2$-tree of order at least~$3$. By induction on the order of $G$ we shall prove that $G$ has a perfect $3$-cover and that $G$ belongs to the family $\cal E$. It is straightforward to observe that the implications $(\rm c) \Rightarrow (\rm a)$ and $(\rm c) \Rightarrow (\rm b)$ are true if $G$ is a $2$-tree of order~$n \le 6$. Now let $G$ be an $\alpha$-excellent $2$-tree of order greater than~$6$ and assume that the  implications $(\rm c) \Rightarrow (\rm a)$ and $(\rm c) \Rightarrow (\rm b)$ are true for smaller $\alpha$-excellent $2$-trees. Let $(T_1,T_2,\ldots,T_p)$ be a~longest $3$-path in $G$, that is, a longest sequence $T_1,T_2,\ldots,T_p$ of triangles in $G$, where  $|V({T_i}) \cap V({T_j})|=2$ if $|i-j|=1$, and  $|V({T_i})\cap V({T_j})|\le 1$ if $|i-j|\ge 2$ ($i,j\in [p]$). From the fact that $(T_1,T_2,\ldots,T_p)$ is a~longest $3$-path in $G$ (which is an $\alpha$-excellent $2$-tree of order at least~$7$) it follows that $p\ge 4$. Assume that $a$, $b$, $c$, $d$, and $e$ are vertices of $G$ for which  $V({T_{p-3}}) \cap V({T_{p-2}})= \{a,b\}$, $V({T_{p-2}})\setminus V({T_{p-3}})= \{c\}$, $V({T_{p-1}})\setminus V({T_{p-2}})= \{d\}$, and $V({T_{p}})\setminus  V({T_{p-1}})= \{e\}$. From the choice of $(T_1,\ldots ,T_p)$, the vertex~$e$ is of degree~$2$ and $T_p$ is a simplex in $G$. Let $G_0', G_1',\ldots, G_{\ell}'$ be the components of $G-\{a,b\}$, where $G_0'$ is that component which contains at least one vertex of $T_1$. It is clear that $\ell$ is positive integer.

We now let $G_i$ denote the subgraph of $G$ induced by $V({G_i'})\cup\{a,b\}$ for $i\in \{0\}\cup [\ell]$. Among the graphs $G_1, \ldots, G_{\ell}$, let $H$ be the graph that contains the triangles of the $3$-path ${\cal P}_0= (T_{p-2}, T_{p-1}^0,$ $T_p^0)$, where $T_{p-1}^0=T_{p-1}$ and $T_p^0=T_p$.  It is obvious that if $T_{p-2}$, $T_{p-1}$, and $T_p$ are the only triangles of $H$, then it is possible that $H$ is one of the graphs $H_1$, $H_1'$, $H_2$, and $H_2'$ shown in Fig. \ref{2-trees-Fig-2}. If $H$ contains the triangle $T_{p-2}$, $T_{p-1}$, $T_p$ and a~simplex $T_{p-1}'$ that shares an edge with the triangle $T_{p-2}$ in $H$, then it follows readily from Proposition \ref{simplicies-are-pairwise-vertex-disjoint} that $H$ is one of the graphs $H_3$ or $H_3'$ in Fig. \ref{2-trees-Fig-2}.

\begin{figure}[h!] \begin{center} \hspace*{2ex}\includegraphics[scale=0.95]{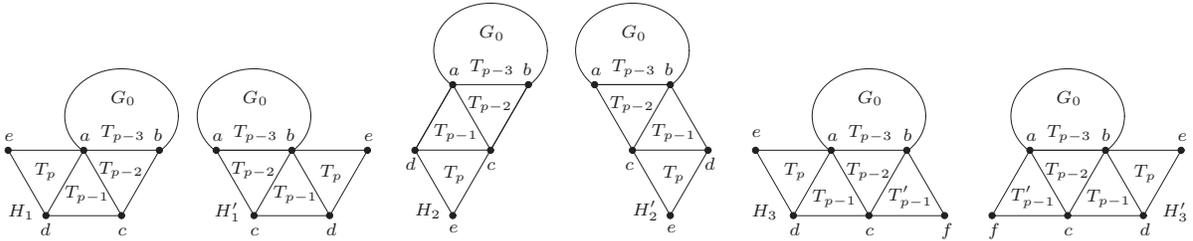}
\caption{Graphs $H_1$, $H_1'$, $H_2$, $H_2'$, $H_3$, and $H_3'$} \label{2-trees-Fig-2} \end{center}\end{figure}

Thus assume that no simplex of $H$ shares an edge with the triangle $T_{p-2}$.
In this case, $H$ is a subgraph induced by the triangles belonging to ${\cal P}_0$ and to some additional $3$-paths ${\cal P}_i=(T_{p-2}, T_{p-1}^i,T_p^i)$, where $i\in [n]$ and $n$ is a positive integer. It follows from Proposition \ref{simplicies-are-pairwise-vertex-disjoint} that if $n=1$, then $H$ is isomorphic to one of the graphs $H_4$, $H_5$ or $H_6$ in Fig. \ref{2-trees-zagubiony}.  Similarly, if $n=2$, then, as can easily be verified, $H$ is isomorphic to the graph $H_7$ shown in Fig. \ref{2-trees-zagubiony}.
Finally, we claim that the case $n\ge 3$ is impossible. Suppose, to the contrary,
that $n\ge 3$. Then let us first observe that if one of the edges $ac$ and $bc$ of the triangle $T_{p-2}$ belongs to at least~$3$ of the triangles $T_{p-1}^0=T_{p-1}$, $T_{p-1}^1, \ldots, T_{p-1}^n$, say to $T_{p-1}^i, T_{p-1}^j, T_{p-1}^k$ (where $0\le i<j<k\le n$), then at least two of the simplexes $T_{p}^i, T_{p}^j, T_{p}^k$ of $H$ (and of $G$) have a common vertex, which is impossible in an $\alpha$-excellent graph $G$. Thus assume that neither $ac$ nor $bc$ belongs to three of the triangles $T_{p-1}^0=T_{p-1}$, $T_{p-1}^1, \ldots, T_{p-1}^n$. Then necessarily $n=3$ and each of the edges $ac$ and $bc$ belongs to exactly two of the triangles $T_{p-1}^0, T_{p-1}^1, T_{p-1}^2, T_{p-1}^3$, say $ac$ belongs to $T_{p-1}^0$ and $T_{p-1}^1$, while $bc$ belongs to $T_{p-1}^2$ and $T_{p-1}^3$. Therefore, as it is easy to verify, if neither the simplexes $T_{p}^0$ and $T_{p}^1$ have a common vertex nor the simplexes $T_{p}^2$ and $T_{p}^3$ have a common vertex, then one of the simplexes $T_{p}^0$ and $T_{p}^1$ has a~common vertex with one of the simplexes $T_{p}^2$ and $T_{p}^3$, which again by Proposition~\ref{simplicies-are-pairwise-vertex-disjoint} is impossible as $G$ is an $\alpha$-excellent graph. This proves that the case $n\ge 3$ is impossible.

\begin{figure}[h!] \begin{center} \hspace*{2ex}\includegraphics[scale=1.0]{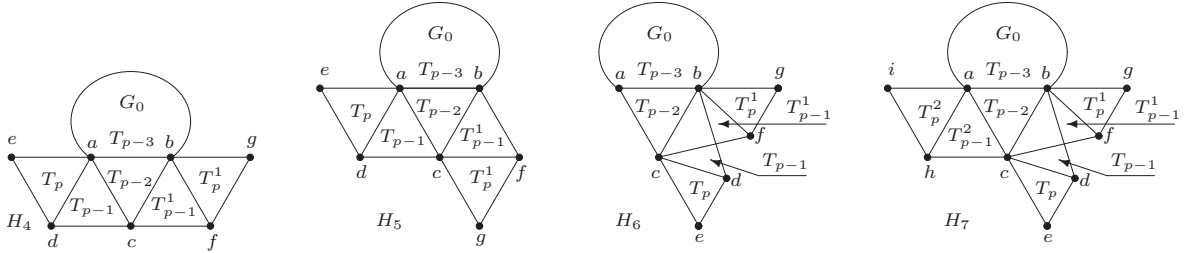}
\caption{Graphs $H_4$, $H_5$, $H_6$, and $H_7$} \label{2-trees-zagubiony} \end{center}\end{figure}

There are now several cases to consider depending on the structure of $H$. We begin showing that $H$ cannot be any of the graphs $H_1$, $H_1'$, $H_4$. Suppose, to the contrary, that $H=H_1$ (where $H_1$ is as illustrated in Fig. \ref{2-trees-Fig-2}). Then, since $G$ is an $\alpha$-excellent graph, there exists an $\alpha$-set $I$ in $G$ that contains $a$. However in this case, $(I \setminus\{a\}) \cup \{c,e\}$ is an independent set in $G$, and so $\alpha(G) \ge |(I \setminus\{a\}) \cup \{c,e\}| > |I| = \alpha(G)$, a contradiction which proves that $H\ne H_1$. Analogously, $H \ne H_1'$. Similarly, suppose that $H=H_4$ (where $H_4$ is shown in Fig. \ref{2-trees-zagubiony}). Let $I_d$ be an $\alpha$-set that contains $d$ in $G$. We note that $|I_d\cap \{b,f,g\}|=1$. If $b\in I_d$, then let $I_d' = (I_d\setminus \{b,d\})\cup \{c,e,g\}$. If $f \in I_d$, then let $I_d' = (I_d\setminus \{d,f\})\cup \{c,e,g\}$. If $g \in I_d$, then let $I_d' = (I_d\setminus \{d\})\cup \{c,e\}$. In all cases, the resulting set $I_d'$ is an independent set in $G$, and so $\alpha(G) \ge |I_d'| > |I_d| = \alpha(G)$, a contradiction. Hence, $H \ne H_4$.

In each of the next five cases (corresponding to the possible graphs $H$, that is, to the graphs $H_2$ (and $H_2'$), $H_3$, $H_5$, $H_6$, and $H_7$) we prove that $G$ belongs to the family $\cal E$ and has a perfect $3$-cover.

\emph{Case 1. $H=H_2$ or $H=H_2'$.}
Without loss of generality, assume that $H=H_2$  (see Fig. \ref{2-trees-Fig-2}). In this case, let $G'$ denote the subgraph $G-\{c,d,e\}$ of $G$. In the next three claims we explain the main relations between properties of the graphs $G$ and $G'$, that is, we prove that: (1) $\alpha(G')=\alpha(G)-1$; (2) $G'$ is an $\alpha$-excellent graph; (3)
$G$ has a perfect $3$-cover and belongs to the family $\cal E$.

\emph{Claim 1.1.} \label{Claim 1.1} $\alpha(G')=\alpha(G)-1$.

If $I'$ is an $\alpha$-set of $G'$, then $I'\cup\{e\}$ is an independent set in $G$
(as $N_G[e]\cap I'= \{c,d,e\}\cap I'= \emptyset$) and therefore $\alpha(G)\ge |I'\cup\{e\}|= \alpha(G')+1$. On the other hand, let $I$ be an $\alpha$-set of $G$. In this case, $|I\cap \{c,d,e\}|=1$ and $I \setminus  \{c,d,e\}$ is independent in $G'$. Thus, $\alpha(G') \ge |I\setminus \{c,d,e\}|=|I|-1=\alpha(G)-1$. Consequently, $\alpha(G')=\alpha(G)-1$.

\emph{Claim 1.2.} $G'$ is an $\alpha$-excellent graph.

Let $v$ be an arbitrary vertex of $G'$. We show that $v$ belongs to an $\alpha$-set of $G'$. Since $\alpha(G')=\alpha(G)-1$ (by Claim 1.1), it suffices to show that $v$ belongs to an independent set of cardinality $\alpha(G)-1=\alpha(G')$ in $G'$. Let $I_v$ be an $\alpha$-set of $G$ that contains $v$, and let $I_v'$ denote the set $I_v\setminus \{c,d,e\}$. Thus, $v\in I_v'$ and the set $I_v'$ is independent in $G'$. Furthermore, $I_v'$ is an $\alpha$-set of $G'$ noting that $|I_v\cap \{c,d,e\}|=1$ and $|I_v'| = |I_v\setminus \{c,d,e\}|= |I_v|-1= \alpha(G)-1=\alpha(G')$. This proves that $G'$ is an $\alpha$-excellent graph.

\emph{Claim 1.3.} $G$ has a perfect $3$-cover and $G$ belongs to the family $\cal E$.

By Claim 1.2, $G'$ is an $\alpha$-excellent graph. Since the order of $G'$ is less than the order of $G$, applying the induction hypothesis we infer that $G'$ has a perfect $3$-cover and $G'$  belongs to the family $\cal E$. Now, if ${\cal P}'$ is a~perfect $3$-cover of $G'$, then ${\cal P}'\cup\{cde\}$ is a perfect $3$-cover of $G$. In addition, since $G'$ belongs to the family $\cal E$, the graph $G'$ can be obtained recursively from a red triangle by operations ${\cal O}_1$ and ${\cal O}_2$, and, since $G$ can be obtained from $G'$ by the operation ${\cal O}_1={\cal O}_1(a,b)$ (that is, by ${\cal O}_1$ applied to the edge $ab$ of $G'$), the graph $G$ can be obtained recursively from a red triangle by operations ${\cal O}_1$ and ${\cal O}_2$. Thus, $G$ belongs to the family $\cal E$.

\emph{Case 2. $H=H_3$ or $H=H_3'$.}
Without loss of generality, assume that $H=H_3$  (see Fig. \ref{2-trees-Fig-2}). This time, let $G'$ denote the subgraph $G-\{d,e,f\}$ of $G$. As in Case 1, we study desired relations between properties of the graphs $G$ and $G'$.

\emph{Claim 2.1.} $\alpha(G')=\alpha(G)-1$.

Let $I'$ be an $\alpha$-set of $G'$. In this case, $|I'\cap \{a,b,c\}|=1$ and either $a\in I'$ or $\{b,c\}\cap I'\ne \emptyset$. Consequently, either $I'\cup\{f\}$ or $I'\cup \{e\}$ is an independent set in $G$, respectively, and therefore $\alpha(G)\ge |I'\cup \{f\}|= |I'\cup\{e\}|= \alpha(G')+1$. This proves that $\alpha(G)\ge \alpha(G')+1$. Now, let $I$ be an $\alpha$-set of $G$. Then $|I\cap \{a,b,c\}|\le 1$ and we consider four cases. If $I\cap \{a,b,c\}=\emptyset$, then $|I\cap \{d,e,f\}|=2$ and $(I\setminus \{d,e,f\}) \cup \{c\}$ is an independent set of cardinality $\alpha(G)-1$ in $G'$. If $a\in I$, then $f\in I$, and $I\setminus \{f\}$ is an independent set of cardinality $\alpha(G)-1$ in $G'$. If $b\in I$, then $|I\cap \{d,e\}|=1$, $f\notin I$, and $I\setminus \{d,e\}$ is an independent set of cardinality $\alpha(G)-1$ in $G'$. If $c\in I$, then $e\in I$, $f \notin I$, and $I\setminus \{e\}$ is an independent set of cardinality $\alpha(G)-1$ in $G'$. Consequently, $\alpha(G')\ge\alpha(G)-1$. This proves that $\alpha(G')=\alpha(G)-1$.

\emph{Claim 2.2.} $G'$ is an $\alpha$-excellent graph.

Let $v$ be an arbitrary vertex of $G'$. As in the proof of Claim 1.2, it suffices to show that $v$ belongs to an $\alpha$-set of $G'$, that is, to an independent set of cardinality $\alpha(G')=\alpha(G)-1$ in $G'$. Since $G$ is $\alpha$-excellent, every vertex of $G$ belongs to some $\alpha$-set in $G$. Let $I_x$ be an $\alpha$-set of $G$ that contains $x$, where $x\in V(G)$. We note that $I_a \cap \{d,e,f\}= \{f\}$ and, consequently, $I_a\setminus \{f\}$ is an $\alpha$-set of $G'$ that contains the vertex $a$. Similarly, $I_b\setminus \{d,e\}$ and $I_c \setminus \{e\}$ are $\alpha$-sets of $G'$ and they contain $b$ and $c$, respectively. If $v\in V({G'}) \setminus\{a,b,c\}$ and $I_v\cap \{a,b,c\}=\emptyset$, then $|I_v\cap \{d,e,f\}|=2$ and $(I_v\setminus  \{d,e,f\})\cup \{c\}$ is an $\alpha$-set of $G'$ and it contains $v$. Finally, if $v\in V({G'})\setminus \{a,b,c\}$ and $I_v\cap \{a,b,c\}\ne \emptyset$, then $|I_v\cap \{a,b,c\}|=1$, $|I_v\cap \{d,e,f\}|=1$, and, therefore, $I_v\setminus  \{d,e,f\}$ is an $\alpha$-set of $G'$ and it contains $v$.  Consequently, $G'$ is an $\alpha$-excellent graph.

\emph{Claim 2.3.} $G$ has a perfect $3$-cover and $G$ belongs to the family $\cal E$.

Now, similarly as in the proof of Claim 1.3, since $G'$ is an $\alpha$-excellent graph of order less than the order of~$G$, the induction hypothesis implies that $G'$ has a perfect $3$-cover and that $G'$ belongs to the family $\cal E$. Certainly, if ${\cal P}'$ is a perfect $3$-cover of $G'$, then $abc\in {\cal P}'$ and  $({\cal P}'\setminus \{abc\})\cup\{ade, bcf\}$ is a perfect $3$-cover of $G$. In addition, since $G'$ belongs to the family $\cal E$, the graph $G'$ can be obtained recursively from a~red triangle by operations ${\cal O}_1$ and ${\cal O}_2$. From this and from the obvious fact that $G$ can be obtained from $G'$ by the operation ${\cal O}_2= {\cal O}_2(bc,ac)$ (that is, by ${\cal O}_2$ applied to the edge $bc$ of the simplex $abc$ of $G'$ and to the edge $ac$ incident with the third vertex $c$ of the triangle $abc$), the graph $G$ can be obtained recursively from a red triangle by operations ${\cal O}_1$ and ${\cal O}_2$. Thus, $G$ belongs to the family $\cal E$.

\emph{Case 3. $H=H_5$.} In this case, we let $G'$ denote the subgraph $G-\{e,f,g\}$ of $G$ and, as before, we study desired relations between properties of the graphs $G$ and $G'$.

\emph{Claim 3.1.} $\alpha(G')= \alpha(G)-1$.

Let $I'$ be an $\alpha$-set of $G'$. In this case, $|I' \cap \{a,c,d\}|=1$. If $a \in I'$ or $d \in I'$, then the set $I'\cup \{g\}$ is an independent set of $G$, while if $c \in I'$, then the set $I'\cup \{e\}$ is an independent set of $G$, implying that $\alpha(G) \ge \alpha(G') + 1$. It remains to prove that $\alpha(G') \ge \alpha(G)-1$. To prove this, let $I$ be an $\alpha$-set of $G$. Thus, $|I\cap \{a,d,e\}|=1$ and $|I \cap \{c,f, g\}|=1$. If $a \in I$ or $d \in I$, then $|I \cap \{f,g\}|=1$ and we let $I' = I \setminus  \{f,g\}$. If $e \in I$ and $c \in I$, then we let $I' = I \setminus  \{e\}$. If $e \in I$ and $b \in I$, then $g \in I$ and we let $I' = (I\setminus \{e,g\}) \cup \{d\}$. If $e \in I$ and $I \cap \{b,c\} = \emptyset$, then $|I \cap \{f,g\}|=1$ and we let $I' = (I\setminus \{e,f,g\}) \cup \{c\}$. In all the above cases, the set $I'$ is an independent set in $G'$ and $|I'| = |I| - 1$, implying that $\alpha(G') \ge |I'| = |I| - 1 = \alpha(G) - 1$. As observed earlier, $\alpha(G') \le \alpha(G)-1$. Consequently, $\alpha(G') = \alpha(G)-1$.

\emph{Claim 3.2.}
$G'$ is an $\alpha$-excellent graph.

Since $G$ is an $\alpha$-excellent graph, every vertex of $G$ belongs to an $\alpha$-set of $G$. Let $I_x$ denote an $\alpha$-set of $G$ that contains the vertex $x$ of $G$. Let $v$ be an arbitrary vertex of $G'$. We show that $v$ belongs to an $\alpha$-set of $G'$, that is, to an independent set of cardinality $\alpha(G')=\alpha(G)-1$ in $G'$. We note that $|I_v \cap \{a,d,e\}|=1$, $|I_v \cap \{c,f,g\}|=1$, and $|I_v \cap \{a,b,c,d\}|\in \{0,1,2\}$. If $|I_v \cap \{a,b,c,d\}|=2$, then $I_v \cap \{a,b,c,d\}=\{b,d\}$, $g\in I_v$ and we let $I_v'=I_v\setminus \{g\}$. If $I_v \cap \{a,b,c,d\}=\emptyset$, then $e\in I_v$ and $|I_v\cap \{f,g\}|=1$ and we let $I_v'= (I_v\setminus\{e,f,g\})\cup \{c\}$. If $I_v \cap \{a,b,c,d\}=\{a\}$ or $I_v \cap \{a,b,c,d\}=\{d\}$, then $|I_v\cap \{f,g\}|=1$ and we let $I_v'=I_v\setminus\{f,g\}$. If $I_v \cap \{a,b,c,d\}=\{b\}$, then $e\in I_v$, $g\in I_v$, and we let $I_v'=(I_v\setminus\{e,g\})\cup \{d\}$.
If $I_v \cap \{a,b,c,d\}=\{c\}$, then $e\in I_v$ and we let $I_v'=I_v\setminus\{e\}$.
In all cases, the set $I_v'$ is an independent set in $G'$ that contains the vertex~$v$ and $|I_v'| = |I_v| - 1 = \alpha(G) - 1 = \alpha(G')$, implying that $v$ belongs to an $\alpha$-set of $G'$, as desired.

\emph{Claim 3.3.} $G$ has a perfect $3$-cover and $G$ belongs to the family $\cal E$.

Similarly as in the proofs of Claims 1.3 and 2.3, since $G'$ is an $\alpha$-excellent graph of order less than the order of~$G$, the induction hypothesis implies that $G'$ has a perfect $3$-cover and that $G'$ belongs to the family $\cal E$. Now, if ${\cal P}'$ is a perfect $3$-cover of $G'$, then $acd \in {\cal P}'$ and  $({\cal P}' \setminus \{acd\})\cup\{ade, cfg\}$ is a perfect $3$-cover of $G$. In addition, since $G'$ belongs to the family $\cal E$, the graph $G'$ can be obtained recursively from a red triangle by operations ${\cal O}_1$ and ${\cal O}_2$,  and, since $G$ can be obtained from $G'$ by the operation ${\cal O}_2= {\cal O}_2(ad,cb)$ (that is, by ${\cal O}_2$ applied to the edge $ad$ of the simplex $acd$ of $G'$ and to the edge $cb$ incident with the third vertex $c$ of the triangle $acd$), the graph $G$ can be obtained recursively from a red triangle by operations ${\cal O}_1$ and ${\cal O}_2$. Thus, $G$ belongs to the family $\cal E$.

\emph{Case 4. $H= H_6$.} In this case, we let $G'$ denote the subgraph $G-\{e,f,g\}$ of $G$ and, as before, we study desired relations between properties of the graphs $G$ and $G'$.

\emph{Claim 4.1.} $\alpha(G')= \alpha(G)-1$.

Let $I'$ be an $\alpha$-set of $G'$. In this case, $|I' \cap \{b,c,d\}|=1$.  If $c \in I'$ or $d \in I'$, then the set $I' \cup \{g\}$ is an independent set of $G$, while if $b \in I'$, then the set $I' \cup \{e\}$ is an independent set of $G$, implying that $\alpha(G) \ge \alpha(G') + 1$. It remains to prove that $\alpha(G') \ge \alpha(G)-1$. By supposition, $G$ is an $\alpha$-excellent graph, and so every vertex of $G$ belongs to some $\alpha$-set of $G$. Let $I$ be an $\alpha$-set of $G$ that contains the vertex~$c$. Necessarily, $g \in I$ and $\{b, e,f\} \cap I = \emptyset$. Thus, $I \setminus \{g\}$ is an independent set in $G'$, implying that $\alpha(G') \ge |I| - 1 = \alpha(G) - 1$. As observed earlier, $\alpha(G') \le \alpha(G)-1$. Consequently, $\alpha(G') = \alpha(G)-1$.

\emph{Claim 4.2.} $G'$ is an $\alpha$-excellent graph.

Since $G$ is an $\alpha$-excellent graph, every vertex of $G$ belongs to an $\alpha$-set of $G$. Let $I_x$ denote an $\alpha$-set of $G$ that contains the vertex $x$ of $G$. Let $v$ be an arbitrary vertex of $G'$. We show that $v$ belongs to an $\alpha$-set of $G'$, that is, to an independent set of cardinality $\alpha(G')=\alpha(G)-1$ in $G'$. This time we note that $|I_v \cap \{c,d,e\}|=1$, $|I_v \cap \{b,f,g\}|=1$, and $|I_v \cap \{a,b,c,d\}|\in \{0,1,2\}$. If $|I_v \cap \{a,b,c,d\}|=2$ (that is, if $I_v \cap \{a,b,c,d\}=\{a,d\}$) or $I_v \cap \{a,b,c,d\}=\{d\}$, then $|I_v \cap \{f,g\}| = 1$ and we let $I_v' = I_v \setminus \{f,g\}$. If $I_v \cap \{a,b,c,d\}=\emptyset$ or $I_v \cap \{a,b,c,d\}=\{a\}$, then $e\in I_v$, $|I_v \cap \{f,g\}| = 1$ and we let $I_v' = (I_v \setminus \{e,f,g\})\cup \{d\}$. If  $I_v \cap \{a,b,c,d\}=\{b\}$, then $e\in I_v$ and we let $I_v' = I_v \setminus \{e\}$. If  $I_v \cap \{a,b,c,d\}=\{c\}$, then $g\in I_v$ and we let $I_v' = I_v \setminus \{g\}$. In all cases, the set $I_v'$ is an independent set in $G'$ that contains the vertex~$v$ and $|I_v'| = |I_v| - 1 = \alpha(G) - 1 = \alpha(G')$, implying that $v$ belongs to an $\alpha$-set of $G'$, as desired.

\emph{Claim 4.3.} $G$ has a perfect $3$-cover and $G$ belongs to the family $\cal E$.

Similarly as in the proofs of Claims 1.3, 2.3 and 3.3, since $G'$ is an $\alpha$-excellent graph of order less than the order of~$G$, the induction hypothesis implies that $G'$ has a perfect $3$-cover and that $G'$ belongs to the family $\cal E$. Now, if ${\cal P}'$ is a perfect $3$-cover of $G'$, then $bcd\in {\cal P}'$ and  $({\cal P}' \setminus \{bcd\})\cup\{cde, bfg\}$ is a perfect $3$-cover of $G$. In addition, since $G'$ belongs to the family $\cal E$, the graph $G'$ can be obtained recursively from a red triangle by operations ${\cal O}_1$ and ${\cal O}_2$,  and, since $G$ can be obtained from $G'$ by the operation ${\cal O}_2= {\cal O}_2(cd,bc)$ (that is, by ${\cal O}_2$ applied to the edge $cd$ of the simplex $bcd$ of $G'$ and to the edge $bc$ incident with the third vertex $b$ of the triangle $bcd$), the graph $G$ can be obtained recursively from a red triangle by operations ${\cal O}_1$ and ${\cal O}_2$. Thus, $G$ belongs to the family~$\cal E$.

\emph{Case 5. $H= H_7$.} In this case, we let $G'$ denote the subgraph $G-\{e,f,g\}$ of $G$. Using identical arguments as in the proof of Case 4, we prove that $G$ has a perfect $3$-cover and it belongs to the family $\cal E$. We omit the details.

Thus all possible cases have been considered. This completes the proof of Theorem~\ref{excellent-2-trees}.
\end{proof}

It is easy to observe that the corona graph $H\circ K_1$ (the graph formed from $H$ by adding for each vertex $v$ in $H$ a new vertex $v'$ and the edge $vv'$) is an $\alpha$-excellent graph for every graph $H$. This implies that every graph is an induced subgraph of an $\alpha$-excellent graph. As a consequence of Theorem~\ref{excellent-2-trees}, we have the following property of $2$-trees.

\begin{cor}
\label{cor:main}
Every $2$-tree is an induced subgraph of an $\alpha$-excellent $2$-tree.
\end{cor}

\begin{proof} Let $G$ be a $2$-tree and let $\cal P$ be a family of vertex-disjoint triangles of $G$. Let $Q$ be the set of all vertices in $G$ that do not belong to any triangle in ${\cal P}$. If the set $Q$ is empty, then, by Theorem \ref{excellent-2-trees}, $G$ is an $\alpha$-excellent $2$-tree itself (and, therefore, it has the desired property). Thus assume that $Q$ is nonempty. For each vertex $v \in Q$, we do the following. Let $v'$ be an arbitrary neighbor of $v$ in $G$. We now add two new vertices $x_{vv'}$ and $y_{vv'}$, and four edges $x_{vv'}v$, $x_{vv'}v'$, $y_{vv'}v$ and $y_{vv'}x_{vv'}$. We note that the newly constructed triangle $vx_{vv'}y_{vv'}$ covers the vertex $v$ and the two new vertices $x_{vv'}$ and $y_{vv'}$. Let $G'$ be a resulting $2$-tree obtained from $G$ by performing this operation for all vertices $v \in Q$. Then the set ${\cal P} \cup \bigcup_{v \in Q} \{vx_{vv'}y_{vv'}\} $ is a perfect $3$-cover of $G'$. Thus, by Theorem~\ref{excellent-2-trees}, the graph $G'$ is an $\alpha$-excellent $2$-tree. By construction, the $2$-tree $G$ is an induced subgraph of the $\alpha$-excellent $2$-tree $G'$. This completes the proof of Corollary~\ref{cor:main}.
\end{proof}

We close this paper with the following open question that we have yet to settle: Is it true that if $k\ge 3$ is an integer and $G$ is an $\alpha$-excellent $k$-tree of order at least $k+1$, then $G$ has a perfect $(k+1)$-cover?

\end{document}